\documentclass{gtpart}     
\usepackage{pinlabel}  
\usepackage{graphicx}  

\usepackage{amsmath}
\usepackage{amssymb}
\usepackage{amsthm}

\title{On Watanabe's theta graph diffeomorphism in the $4$--sphere}

\author{David T Gay}
\givenname{David T}
\surname{Gay}
\address{Department of Mathematics\\ University
  of Georgia\\ Athens, GA 30602}
\email{dgay@uga.edu}

\volumenumber{}
\issuenumber{}
\publicationyear{}
\papernumber{}
\startpage{}
\endpage{}
\doi{}
\MR{}
\Zbl{}
\received{}
\revised{}
\accepted{}
\published{}
\publishedonline{}
\proposed{}
\seconded{}
\corresponding{}
\editor{}
\version{}

\newtheorem{theorem}{Theorem}

\theoremstyle{definition}

\def\Z{\mathbb Z}
\def\N{\mathbb N}

\def\R{\mathbb R}

\def\cH{\mathcal{H}}

\newcommand{\into}{\ensuremath{\hookrightarrow}}

\newcommand{\Diff}{\mathop{\rm Diff}\nolimits}
\newcommand{\Emb}{\mathop{\rm Emb}\nolimits}

\newcommand{\Mod}{\mathop{\rm Mod}\nolimits}
\newcommand{\wat}{\mathop{\rm wat}\nolimits}

\begin{document}

\begin{abstract}    
Watanabe's theta graph diffeomorphism, constructed using Watanabe's clasper surgery construction~\cite{WatanabeExotic} which turns trivalent graphs in $4$--manifolds into parameterized families of diffeomorphisms of $4$--manifolds, is a diffeomorphism of $S^4$ representing a potentially nontrivial smooth mapping class of $S^4$. The ``$(1,2)$--subgroup'' of the smooth mapping class group of $S^4$ is the subgroup represented by diffeomorphisms which are pseudoisotopic to the identity via a Cerf family with only index $1$ and $2$ critical points. This author and Hartman~\cite{GayHartman} showed that this subgroup is either trivial or has order $2$ and explicitly identified a diffeomorphism that would represent the nontrivial element if this subgroup is nontrivial. Here we show that the theta graph diffeomorphism is isotopic to this one possibly nontrivial element of the $(1,2)$--subgroup. To prove this relation we develop a diagrammatic calculus for working in the smooth mapping class group of $S^4$.
\end{abstract}

\maketitle


This paper exhibits a relation in the smooth mapping class group of $S^4$, which we denote $\Mod(S^4) = \pi_0(\Diff^+(S^4))$, where $\Diff^+(S^4)$ is the group of orientation preserving diffeomorphisms of $S^4$. In other words, we will exhibit a smooth isotopy from one diffeomorphism of $S^4$ to another. Part of the point is that the two diffeomorphisms involved are each of independent interest and arise in different contexts and thus exhibiting this relation improves our understanding of both diffeomorphisms. Another point is to illustrate new techniques for understanding relations in the smooth mapping class groups of $4$--manifolds which have the potential to be useful in other contexts.

It is important to emphasize that we do not at this time know whether or not $\Mod(S^4)$ is trivial. Of course, if it is trivial then our relation is trivially true, but it is possible that either $\Mod(S^4)$ is nontrivial or that to prove that $\Mod(S^4)$ is trivial requires tools similar to those developed in this paper.

The relation we prove has the form
\[ \wat(\Theta) = \cH(\alpha(1)) \in \Mod(S^4), \]
so we now explain what the left and right hand sides of the equation mean. Briefly, for those for whom this is meaningful:
\begin{itemize}
 \item $\Theta$ is the theta graph embedded in $S^4$,
 \item $\wat(\cdot)$ denotes Watanabe's clasper surgery construction~\cite{WatanabeExotic} for producing families of diffeomorphisms from trivalent graphs,
 \item $\wat(\Theta)$ is (the smooth isotopy class of) Watanabe's theta graph diffeomorphism~\cite{WatanabeExotic},
 \item $\alpha(1)$ is the first in a countable list $\alpha(i)$ of generators of $\pi_1(\Emb(S^1,S^1 \times S^3))$ established by Budney and Gabai in~\cite{BudneyGabai}, where $\Emb(A,B)$ denotes the space of embeddings of $A$ in $B$,
 \item $\cH$ is a Cerf-theoretic homomorphism from $\pi_1(\Emb(S^1,S^1 \times S^3))$ to $\pi_0(\Diff^+(S^4))$ described in~\cite{GayDiffeos}, and
 \item $\cH(\alpha(1))$ is shown by this author and Hartman in~\cite{GayHartman} to have order at most $2$ and to generate the subgroup $\Mod^{1,2}(S^4)$ of $\Mod(S^4)$ of all diffeomorphisms which are pseudoisotopic to the identity via a pseudoisotopy realizable by a Cerf family involving only critical points of index $1$ and $2$.
\end{itemize}

This relation is also proved by Kosanovi\'c as Corollary~1.11 in~\cite{Kosanovic}. The current paper was completed after the initial posting of~\cite{Kosanovic} and so is best interpreted as an alternative proof of Kosanovi\'c's result, naturally building on many similar techniques, and showcasing a diagrammatic toolbox. In Kosanovi\'c's notation, what we call $\cH(\alpha(1))$ is denoted $W(1)$.

Given an oriented trivalent graph $\Gamma$ with $2k$ vertices embedded in a $4$--manifold $X$, with the property that no vertex has all incoming or all outgoing edges, Watanabe~\cite{WatanabeExotic} describes a ``clasper surgery'' construction which produces an element $\wat(\Gamma) \in \pi_{k-1}(\Diff^+(X))$. Note that Watanabe does not use the notation $\wat(\cdot)$, this is notation we have adopted following~\cite{Kosanovic}. For certain values of $k$ and certain graphs $\Gamma \subset X = B^4$, but {\em not including} the case $k=1$, Watanabe is then able to show that $\wat(\Gamma)$ is nontrivial, thus disproving the smooth $4$--dimensional Smale conjecture. In the case $k=1$ the only interesting graph to consider is the theta graph, which we denote $\Theta$, and thus a fundamental question related to Watanabe's work is whether or not $\wat(\Theta) \in \Mod(B^4)$ is trivial. Note that all embeddings of $\Theta$ in a simply connected $4$--manifold are isotopic so we do not need to specify the embedding. Also note that $\Mod(B^4) = \Mod(S^4)$, so we will work in $S^4$ rather than $B^4$.

Of course, if one shows that $\wat(\Theta)$ is nontrivial one will have shown the dramatic result that $\Mod(S^4)$ is nontrivial, but if one shows that $\wat(\Theta)$ is trivial one has only shown that the clasper surgery construction fails to detect nontriviality of $\Mod(S^4)$. Here we thread the needle between these two possibilities, showing that $\wat(\Theta)$ is equal to another mapping class $\cH(\alpha(1))$ about which a little more is known.

Instead of studying diffeomorphisms up to isotopy one may first study diffeomorphisms up to {\em pseudoisotopy}: A diffeomorphism $\phi:X \to X$ is pseudoisotopic to the identity if there exists a diffeomorphism $\Phi: [0,1] \times X \to [0,1] \times X$ which is the identity on $\{0\} \times X$ and $[0,1] \times \partial X$ and equals $\phi$ on $\{1\} \times X$. If $\Phi$ is level preserving then $\Phi$ is actually an isotopy. In fact every diffeomorphism of $S^4$ is pseudoisotopic to the identity; see~\cite{GayDiffeos} for a proof that this follows easily from the fact~\cite{KervaireMilnor} that there are no exotic $5$--spheres.

A pseudoisotopy $\Phi : [0,1] \times X \to [0,1] \times X$ gives rise to a pair of Morse functions $g_0,g_1: [0,1] \times X \to [0,1]$, where $g_0$ is simply projection to $[0,1]$ and $g_1 = g_0 \circ \Phi$. Both $g_0$ and $g_1$ are Morse functions without critical points. Cerf's key observation is that if these can be connected by a homotopy $g_t$ such that each $g_t$ does not have critical points then $\Phi$ is isotopic to an isotopy, and thus that the problem of turning pseudosotopies into isotopies can be turned into the problem of simplifying $1$--parameter families, called ``Cerf families'', of Morse functions (allowing births and deaths) between Morse functions without critical points. In our case, when the dimension of $X$ is $4$ and $X$ is simply connected, one can always arrange that the critical points that do arise are only of index $2$ and $3$~\cite{HatcherWagoner}, and in special cases one can arrange that they are only of index $1$ and $2$. When the critical points are only of index $1$ and $2$, the Cerf family is significantly easier to study.

The subgroup of $\Mod(S^4)$ represented by diffeomorphisms which are pseudoisotopic to the identity via a pseudoisotopy with a Cerf family involving only critical points of index $1$ and $2$ is known as ``the $(1,2)$ subgroup'' and denoted $\Mod^{1,2}(S^4)$. In~\cite{GayDiffeos} this author exhibited a group homomorphism $\cH: \pi_1(\Emb(S^1,S^1 \times S^3),*) \to \Mod(S^4)$ and showed that the image is exactly $\Mod^{1,2}(S^4)$. The basepoint $*$ is the embedding of $S^1$ as $S^1 \times \{p\}$ for some $p \in S^3$. Budney and Gabai computed $\pi_1(\Emb(S^1,S^1 \times S^3),*)$ and showed that it is free abelian with a countable, explicitly described, list of generators $\alpha(i)$, $i \in \N$. This author and Hartman~\cite{GayHartman} then showed that in fact $\Mod^{1,2}(S^4)$ is generated just by $\cH(\alpha(1))$ and that $\cH(\alpha(1))$ is either trivial or of order $2$.

From the description of $\wat(\Theta)$ which we will give in the next section, or from a general discussion of pseudosotopies and Watanabe's constructions by Botvinnik and Watanabe~\cite{BotvinnikWatanabe}, it is not hard to see that $\wat(\Theta) \in \Mod^{1,2}(S^4)$. Thus, expressing $\wat(\Theta)$ in terms of $\cH(\alpha(1)$ should be a step in the direction of resolving the triviality or nontriviality of $\wat(\Theta)$.

\begin{theorem} \label{T:Main}
 Watanabe's theta graph diffeomorphism $\wat(\Theta) \in S^4$ is trivial if and only if the $(1,2)$--subgroup $\Mod^{1,2}(S^4)$ is trivial. In particular, $\wat(\Theta) = \cH(\alpha(1))$, and $\cH(\alpha(1))$ is known to be of order at most $2$ and to generate $\Mod^{1,2}(S^4)$.
\end{theorem}

\section{Diagrammatic notation} \label{S:Diagrams}

Central to our argument is something like a Kirby calculus for families of $5$--dimensional handlebodies. At the moment this is far from being a complete calculus, but rather an ad hoc collection of tools that works in the context at hand. Consider the two diagrams in \autoref{F:BeginningAndEnd}. The diagram on the left is an illustration of $\wat(\Theta)$ and the diagram on the right is an illustration of $\cH(\alpha(1))$; these two statements will be justified in Section~\ref{S:BeginningAndEnd}. The purpose of this paper is to show how to get from the left diagram to the right diagram using various moves. First we explain what both of these diagrams, and the intermediate diagrams in our moves, mean.

\begin{figure}
 \labellist
 \small\hair 2pt
 \pinlabel {$\wat(\Theta) =$} [r] at 0 50
 \pinlabel {$=$} at 157 50
 \pinlabel {$= \cH(\alpha(1))$} [l] at 260 50
 \endlabellist
 \centering
 
 \includegraphics[width=.6\textwidth]{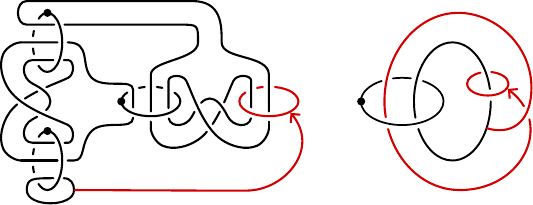}
 \caption{The beginning and the end of our forthcoming sequence of diagrams. Section~\ref{S:Diagrams} explains the meaning of the two diagrams in this figure, Section~\ref{S:BeginningAndEnd} justifies the left-most and right-most equals signs, and Section~\ref{S:Sequence} gives a sequence of intermediate diagrams justifying the middle equals sign. \label{F:BeginningAndEnd}}
\end{figure}

Each diagram is meant to illustrate a collection of unknotted $S^2$'s in $S^4$, bounding obvious $B^3$'s, together with a $1$--parameter family of embeddings of some number of $S^1$'s in the complement of these $S^2$'s. After explaining how to see this, we will show how such an embedding gives rise to a smooth mapping class of $S^4$.

Each figure is drawn in $\R^3$ seen as a slice $\{(x_1,x_2,x_3,x_4) \mid x_4=0\}$ of $\R^4 = \{(x_1,x_2,x_3,x_4)\}$, which in turn sits in $S^4 = \R^4 \cup \{\infty\}$. The figures have four diagrammatic elements: (1) dotted circles, always unknotted and drawn in black, (2) possibly knotted and linked black circles, without dots, (3) red unknotted circles and (4) red arrows, each pointing from one of the undotted black circles to one of the red circles. Each red circle should have exactly one red arrow pointing to it, but each black undotted circle can have arbitrarily many red arrows attached to it. We explain each of these as follows:
\begin{enumerate}
 \item Each black dotted circle bounds an obvious flat disk, and this disk shrinks to a point in $x_4 <0$ and $x_4 > 0$, thus describing a $3$--ball and its boundary $S^2$. Abusing terminology a bit, we henceforth refer to the black dotted circles themselves as ``dotted $S^2$'s''.
 \item Each black undotted circle is the starting and ending position of a $t$--parameterized family, for $t \in [0,1]$, of embeddings of $S^1$ into $S^4$. In other words, for each undotted black circle, we are working towards describing a family of embeddings $\alpha_t: S^1 \into S^4$ such that the drawn black circle is $\alpha_0(S^1) = \alpha_1(S^1)$. For this reason we call these black undotted circles ``basepoint circles''.
 \item Each red circle represents an embedded $S^2$ in $S^4$ given in the usual way: the circle illustrated lies in $\{x_4=0\}$ and shrinks to a point for $x_4>0$ and $x_4 < 0$. We call each of these red $S^2$'s a ``spinning sphere'', for reasons which should be clear shortly. This $S^2$ is actually the trace of a $t$--parameterized family of arcs, for $t \in [1/4,3/4]$. To see how to interpret this $S^2$ as a family of arcs, we first declare that at times $t=1/4$ and $t=3/4$ the arc is a small interval neighborhood in the illustrated red circle centered on the point at which the red arrow meets the red circle. Reparameterizing the $S^2$ so that this short arc is a line of meridian from north to south pole, then as $t$ ranges from $1/4$ to $3/4$ this arc spins around the sphere from west to east. We illustrate this in \autoref{F:SpinningSphere} by redrawing a local picture so that the spinning sphere lies entirely in $\{x_3=0\}$ and arcs of basepoint circles that link the spinning sphere appear as points in the interior of the sphere. One should properly adopt an orientation convention, so that an orientation on the red circle unambiguously determines the direction in which the arc spins around the spinning sphere, but in the argument in this paper this will not be important. In other words, our results hold no matter which orientation we choose.
 \item Each arrow represents a band leading from a basepoint circle to a spinning sphere. This band is the trace of a $t$--parameterized family of arcs for $t \in [0,1/4]$ and $t \in [3/4,1]$. At $t=0$ and $t=1$ this arc is a short arc in the basepoint circle centered at the point where the arrow starts, and then as $t$ ranges from $0$ to $1/4$ the arc performs a finger move along the band until it meets the spinning sphere; from $3/4$ to $1$ the arc simply reverses what it did from $0$ to $1/4$. This is illustrated in \autoref{F:FingerAlongArrow}.
\end{enumerate}
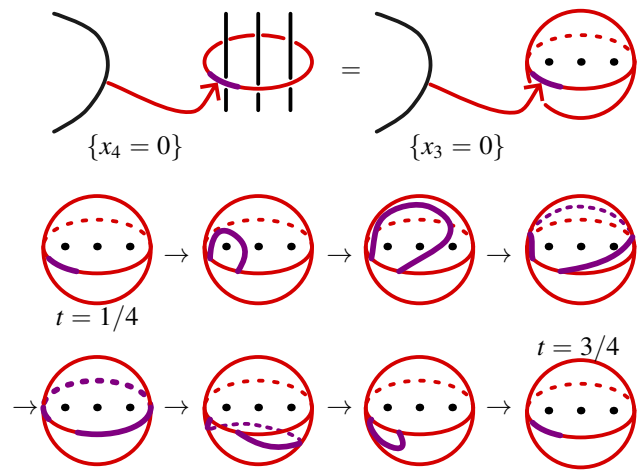
\begin{figure}
 \labellist
 \small\hair 2pt
 \pinlabel $=$ at 82 105
 \pinlabel $\{x_4=0\}$ at 25 85
 \pinlabel $\{x_3=0\}$ at 110 85
 \pinlabel $t=1/4$ at 15 40
 \pinlabel $t=3/4$ at 142 32
 \pinlabel $\to$ at 36 56
 \pinlabel $\to$ at 79 56
 \pinlabel $\to$ at 121 56
 \pinlabel $\to$ at -4 16
 \pinlabel $\to$ at 36 16
 \pinlabel $\to$ at 79 16
 \pinlabel $\to$ at 121 16
 \endlabellist
  \centering
 \includegraphics[width=.6\textwidth]{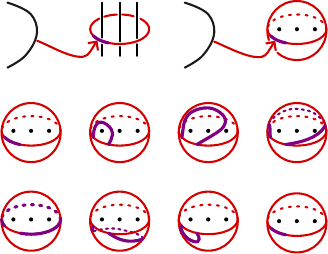}
 \caption{A spinning sphere as a $t$--parameterized family of arcs, for $t \in [1/4,3/4]$. In the top line we first draw the figure in the $3$--dimensional slice $\{x_4=0\}$, so that the sphere appears as a red circle with three linking black arcs, and we then draw the sphere in the slice $\{x_3=0\}$ so that we see the entire sphere, but the arcs now appear as three black dots. \label{F:SpinningSphere}}
\end{figure}
\begin{figure}
 \labellist
 \small\hair 2pt
 \pinlabel $t=0,1$ at 120 37
 \pinlabel {$t \in (0,1/4) \cup (3/4,1)$} at 204 37 
 \pinlabel $t=1/4,3/4$ at 295 37
 \endlabellist
 \centering
 \includegraphics[width=\textwidth]{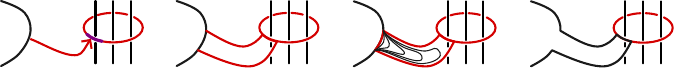}
 \caption{A finger move along an arrow and back. \label{F:FingerAlongArrow}}
\end{figure}
The upshot is that a diagram such as one of the ones shown in \autoref{F:BeginningAndEnd} describes, up to isotopy, an embedding $\beta: \amalg^n B^3 \into S^4$ of some number of $3$--balls into $S^4$, together with a $t$--parameterized loop of embeddings $\alpha_t : \amalg^m S^1 \into S^4 \setminus \beta(\amalg^n \partial B^3)$ of some number of $S^1$'s into the complement of the boundaries of the $3$--balls. The loop $\alpha_t$ is such that each $S^1$ starts as one of the basepoint circles drawn in the diagram and then sticks out a finger along the various red arrows based at that basepoint circle, and then the tip of each such finger spins once around the spinning sphere at the end of the relevant red arrow, and then the finger retracts back along the red arrow.

Note that here we are declaring that when there are multiple spinning spheres, we are still describing a family with one time parameter $t$. One could adopt the convention that each spinning sphere gives its own time parameter, in which case the parameter space would be $(S^1)^k$, where $k$ is the number of spinning spheres. This is in fact how things work with Watanabe's general construction, as described briefly in the next section. However, in the diagrams in this paper all spinning spheres correspond to the same time parameter. As described above, all spinning spheres spin simultaneously in the time interval $t \in [1/4.3/4]$ but it is clear that these spinnings commute with each other and so in fact we can choose to spin around the spheres one at a time in any order, and this will be important in the main calculation of this paper.

First, given $\beta$ and $\alpha_t$, we will build a $6$--manifold $Z$ which is a $5$--manifold bundle over $S^1$, with the fiber over $t \in S^1$ being a $5$--manifold $Z_t$. To do this, let $W^*$ be the result of pushing the interior of each component of $\beta(\amalg^n B^3) \subset S^4 \subset B^5$ into the interior of $B^5$ and carving out a small tubular neighborhood of these $B^3$'s in $B^5$. Thus $W^*$ is diffeomorphic to $\natural^n S^1 \times B^4$, the result of attaching $n$ $1$--handles to $B^5$, and $\partial W^* \cong \#^n S^1 \times S^3$. This is the $5$--dimensional analogue of dotted circle notation in $4$--dimensional Kirby calculus. Now we can view $\alpha_t$ as a family of embeddings of $\amalg^m S^1$ into $\partial W^*$. Extend this to a family of embeddings $\overline{\alpha_t} : \amalg^m S^1 \times B^3 \into \partial Z^*$ using the $0$--framings coming from the original embeddings into $S^4$; here we need to use the fact that $\pi_2(SO(4))$ is trivial in addition to the usual fact that there are two framings of a circle in a $4$--manifold because $\pi_1(SO(4)) = \Z/2\Z$. Let $Z_0 = S^1 \times W^*$. The trace of the family $\overline{\alpha_t}$ can be interpreted as an embedding $\overline{\alpha}: \amalg^m S^1 \times S^1 \times B^3 \into \partial Z_0 = S^1 \times \partial W^*$, with the map on the first $S^1$ factor being the identity. In other words, $\overline{\alpha}(t,p) = (t,\overline{\alpha_t}(p))$. Use this as the gluing map for an $S^1$--parameterized family of $5$--dimensional $1$--handles (equivalently, $6$--dimensional round $1$--handles) attached to $Z_0$, and the resulting $6$--manifold is $Z$. The fiber $Z_t$ over $t \in S^1$ is the $5$--manifold built by attaching $m$ $5$--dimensional $1$--handles to $W^*$ with gluing map $\overline{\alpha_t}$. Since these gluing maps vary smoothly with $t$ and since $\overline{\alpha_1} = \overline{\alpha_0}$, this builds $Z$ as a bundle over $S^1$ with fiber $Z_t$. Furthermore, $\partial Z$ is a bundle over $S^1$ with $4$--dimensional fiber $\partial Z_t$. If we let $X = \partial Z_0 = \partial Z_1$, then the monodromy of this bundle is a well defined element of $\Mod(X)$.

Note that the $5$--dimensional $4$--manifold bundle $\partial Z$ could be described via parameterized $4$--dimensional surgery, rather than via parameterized $5$--dimensional handle attachment, and the arguments in this paper could be rewritten in that language. Kosanovic's presentation~\cite{Kosanovic} is given in terms of surgery. However, perhaps just due to personal limitations, the author of this paper finds it easier to think in terms of handle attachment than in terms of surgery. Also, the surgery description makes the connection to Cerf's approach to pseudoisotopies clearer. Watanabe's description of how to turn a trivalent graph into an element of $\pi_{k-1}(\Diff^+(X))$ is given in terms of surgery, but is easily reinterpreted in terms of handle attachment.

To see that, in our examples, the $4$--manifold $X$ sitting over $t=0$ is diffeomorphic to $S^4$, note that when $t=0$ the arrows and the spinning spheres do not come into play. In the two diagrams in \autoref{F:BeginningAndEnd}, ignoring the arrows and the spinning spheres one sees that the $5$--dimensional $1$-- and $2$--handles all cancel, so that the $5$--manifold $Z_0$ is diffeomorphic to $B^5$. We do not need to worry about whether or not the diffeomorphism from $X = \partial Z_0$ to $S^4$ is canonical because $\Mod(S^4)$ is abelian, and thus the isomorphism from $\Mod(X)$ to $\Mod(S^4)$ given by one diffeomorphism of $X$ to $S^4$ is equal to that given by any other such diffeomorphism.

Returning to \autoref{F:BeginningAndEnd}, we claim that with the above interpretation of the diagrams, the diagram on the left represents $\wat(\Theta)$ and the diagram on the right represents $\cH(\alpha(1))$. In the next section we will explain this, and in the following section we will show how to transform the one diagram into the other using a sequence of moves and intermediate diagrams.

\section{The beginning and end diagrams} \label{S:BeginningAndEnd}
Watanabe's prescription~\cite{WatanabeExotic} for turning an oriented trivalent graph $\Gamma$ in a $4$--manifold $X$ into an element of $\pi_{k-1}(\Diff^(X))$ can be summarized as follows:
\begin{itemize}
 \item Separate the vertices into those of type I, with $2$ incoming edges and $1$ outgoing edge, and those of Type II with $1$ incoming edge and $2$ outgoing edges.
 \item At each type I vertex place a copy of the $4$--dimensional Borromean rings. This is an embedding of $S^2 \amalg S^2 \amalg S^1$ in $B^4$ characterized by the properties that, ignoring the $S^1$, the two $S^2$'s bound disjoint $B^3$'s, and that the $S^1$ reads off the commutator of the meridians of the two $S^2$'s. Arrange for the two incoming edges to land on the two $S^2$'s and for the one outgoing edge to connect to the $S^1$.
 \item At each type II vertex also place a copy of the $4$--dimensional Borromean rings, but now interpret one of the $S^2$'s as a ``spinning sphere'' as in the preceding section, i.e. an $S^1$--parameterized family of arcs spinning around a north and south pole of the $S^2$. Arrange for the one incoming edge to connect to the $S^2$ that does not spin, for one of the outgoing edges to connect to the $S^1$, and for the other outgoing edge to connect to the spinning sphere - this outgoing edge should really be seen as a band which meets the spinning sphere along the initial arc in the sphere's family of arcs.
 \item In the middle of each edge place a copy of the $4$--dimensional Hopf link. This is an unknotted $S^2$ and its linking meridian $S^1$. Arrange that the outgoing half of that edge connects to the $S^2$ and that the incoming half of that edge connects to the $S^1$.
 \item Along the outgoing half of each edge, tube the $S^2$ from the Hopf link at the middle of the edge to the $S^2$ from the Borromean rings at the vertex at the end of the edge.
 \item Along the incoming half of each edge, there is an $S^1$ coming from the Hopf link at the middle of the edge and either an $S^1$ or a spinning sphere from a Borromean rings at the incoming end of the edge. In the $S^1$ case, simply band the two $S^1$'s together. In the spinning sphere case, band the $S^1$ at the Hopf link to the spinning sphere to create an $S^1$--parameterized family of $S^1$'s.
 \item The result is an embedding $\beta$ of $\amalg^e S^2$ into $X$, where $e$ is the number of edges, and an $(S^1)^k$--parameterized family $\alpha_{\mathbf{t}}$, with $\mathbf{t} \in (S^1)^k$, of embeddings of $\amalg^e S^1$ into $X \setminus \beta(\amalg^e S^2)$, where $k$ is the number of type II vertices, which is equal to the number of type I vertices. Here each spinning sphere is interpreted as introducing its own independent parameter.
\end{itemize}

Having done this, Watanabe then wants us to interpret this as an $(S^1)^k$--parameterized family of surgeries along $e$ $S^2$'s and $e$ $S^1$'s, giving a $4$--manifold bundle over $(S^1)^k$. With more work, not relevant when $k=1$, one shows that this bundle is trivial over the $(k-1)$--skeleton of $(S^1)^k$ and thus gives a well defined element of $\pi_{k-1}(\Diff^+(X))$. To turn this $4$--dimensional surgery description into a $5$--dimensional handlebody description, one might think of the $S^2$'s as attaching spheres for $5$--dimensional $3$--handles, but since the $S^2$'s bound standard $B^3$'s which do not move as the parameters vary (only the $S^1$'s move), these $S^2$'s can just as well be interpreted as dotted $S^2$'s describing $5$--dimensional $1$--handle attachments, and this is the perspective we take here.

In the case of the theta graph, $k=1$ so there is only one spinning sphere. Applying the recipe above to the theta graph gives \autoref{F:WTheta}, which uses a slight variation of the notation described in the preceding section, including dotted spheres connected by arcs to indicate spheres that should be tubed together along the arcs.
\begin{figure}
 \labellist
 \Huge\hair 2pt
 \pinlabel $\rightsquigarrow$ at 47 44
 \endlabellist
 \centering
 \includegraphics[width=.8\textwidth]{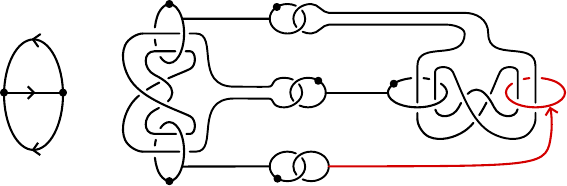}
 \caption{Watanabe's construction applied to the theta graph. \label{F:WTheta}}
\end{figure}
When two dotted spheres are represented by their dotted equatorial circles in our diagrams, and tubed along an arc in the diagram, the resulting sphere is described by the band sum of the two dotted equatorial circles in the diagram. This observation and a simple isotopy turns the diagram on the right of \autoref{F:WTheta} into the diagram on the left of \autoref{F:BeginningAndEnd}. Thus we see that \autoref{F:BeginningAndEnd} does indeed start with a diagrammatic description of $\wat(\Theta)$.

To see that the diagram on the right of \autoref{F:BeginningAndEnd} represents $\cH(\alpha(1))$, first note that the diagram does describe a loop $\alpha_t: S^1 \into S^1 \times S^3$ since surgery along the single dotted $S^2$ turns $S^4$ into $S^1 \times S^3$. With this in mind, this is essentially exactly the same as the description of $\alpha(1)_t$ appearing in~\cite{BudneyGabai},~\cite{GayDiffeos} and~\cite{GayHartman}. The homomorphism $\cH: \pi_1(\Emb(S^1,S^1 \times S^3)) \to \Mod(S^4)$ is nothing more than the same parameterized handle attachment construction described in the preceding section: Think of $S^1 \times S^3$ as the boundary of $S^1 \times B^4$, which is the result of attaching a single $5$--dimensional $1$--handle to $B^5$. A loop of $S^1$'s then gives a loop of attaching data with which to attach a $5$--dimensional $2$--handle, and with this one builds a $5$--manifold bundle over $S^1$ whose boundary is a $4$--manifold bundle over $S^1$ with some monodromy. Because our loop of $S^1$'s starts and ends at $S^1 \times \{p\} \subset S^1 \times S^3$, the fiber of this $4$--manifold bundle is diffeomorphic to $S^4$ and thus the monodromy gives an element of $\Mod(S^4)$. Thus we see that $\cH(\alpha(1))$ is exactly the mapping class described by the figure on the right in \autoref{F:BeginningAndEnd}.

\section{A sequence of diagrams} \label{S:Sequence}
In this section we will show how to transform the parameterized $5$--dimensional handle diagram on the left in \autoref{F:BeginningAndEnd} into the diagram on the right in the same figure. We will use parameterized handle cancellation and handleslides as well as homotopies of our families of embeddings of $S^1$'s.

\begin{proof}[Proof of Theorem~\ref{T:Main}]
 Since the diagram on the left in \autoref{F:BeginningAndEnd} represents $\wat(\Theta)$ and the diagram on the right represents $\cH(\alpha(1))$, by showing how to transform one diagram into the other and showing that the moves we use in this transformation respect the corresponding element of $\Mod(S^4)$, we will complete the proof of Theorem~\ref{T:Main}.
 
 Figure~\ref{F:LabelledHandles} reproduces our initial diagram for $\wat(\Theta)$ and introduces labels $a$, $b$ and $c$ for the three $1$--handles and labels $A$, $B$ and $C$ for the three $2$--handles. Recall that this describes a $t$--parameterized family of $5$--dimensional handlebodies, but that the $2$--handle $A$ is the only $2$--handle that is actually changing with $t$. Also, note that at any time $t$, all handles can easily be cancelled, but that the cancellation requires handleslides and isotopies to arrange that for each $1$--handle there is exactly one $2$--handle which goes over it once, i.e. cancels it, and that the other $2$--handles do not go over that $1$--handle at all. At time $t=0$, when the attaching circle for $2$--handle $A$ is simply the small linking meridian to the dotted $2$--sphere for $1$--handle $a$, this is easy to do: Slide $2$--handle $B$ over $2$--handle $A$ twice to get $B$ off of $a$, then cancel $A$ and $a$, then slide $2$--handle $C$ over $B$ twice to get $C$ off of $b$ and cancel $B$ and $b$, and finally cancel $C$ and $c$. However, once $A$ starts to move, then the slid version of $B$ will move, and so forth, and we will need to be much more careful.
 \begin{figure}
  \labellist
  \small\hair 2pt
  \pinlabel $a$ at 18 37
  \pinlabel $A$ at 12 0
  \pinlabel $b$ at 53 50
  \pinlabel $B$ at 44 40
  \pinlabel $c$ at 18 93
  \pinlabel $C$ at 104 99
  \endlabellist\centering
  \includegraphics[width=.5\textwidth]{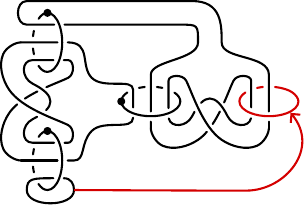}
  \caption{$\wat(\Theta)$ with labelled handles. \label{F:LabelledHandles}}
 \end{figure}
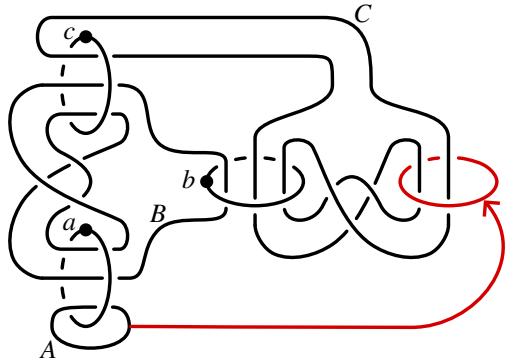

 There are three simplifying moves that we can use to keep the handleslide complexity from exploding too quickly. The first is to observe that, when isotoping $1$--dimensional elements of our diagrams, we can freely change crossings since this isotopy is happening in $\R^4$ not in $\R^3$. This applies to crossings between ``ordinary'' $S^1$'s, crossings between arrows, and crossings between ``ordinary'' $S^1$'s and arrows.
 
 The second move is to observe that a spinning sphere around multiple strands of ``ordinary'' $S^1$'s can be split into multiple spinning spheres, one around each strand, as illustrated in Figure~\ref{F:SplittingSpinningSpheres}.
 \begin{figure}
 \labellist
 \Large\hair 2pt
 \pinlabel $=$ at 80 15
 \endlabellist
 \centering
  \includegraphics[width=.6\textwidth]{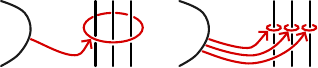}
  \caption{Splitting spinning spheres. \label{F:SplittingSpinningSpheres}}
 \end{figure}
 
 The third move is to exploit the fact that when spinning around mutiple spheres to describe an $S^1$--parameterized family, we can spin around them in any order as discussed earlier. Using the fact that $\Mod(S^4)$ is abelian we can use additive notation in $\Mod(S^4)$ and, interpreting each diagram as an element of $\Mod(S^4)$, we derive the diagrammative move shown in Figure~\ref{F:CommutingSpins}. This move is perhaps better described in words: Given a diagram $D$ with two red arrows $x$ and $y$ pointing to red spinning spheres $X$ and $Y$, respectively, consider the diagram $D_1$ obtained by erasing $x$ and $X$ and the diagram $D_2$ obtained by erasing $y$ and $Y$. Then the mapping class represented by $D$ is the sum of the mapping classes represented by $D_1$ and $D_2$.
 \begin{figure}
 \labellist
 \large\hair 2pt
 \pinlabel $=$ at 108 46
 \pinlabel $+$ at 222 46
 \endlabellist
 \centering
  \includegraphics[width=\textwidth]{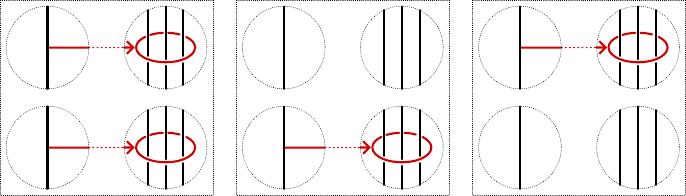}
  \caption{Separating spinning around different spheres in $\Mod(S^4)$. \label{F:CommutingSpins}}
 \end{figure}

 In Figure~\ref{F:SplitASpin} we split the spinning sphere for $2$--handle $A$, then simplify with an isotopy, and then decompose $\wat(\Theta)$ as a sum of two mapping classes, one given by the diagram on the left at the bottom and one given by the diagram on the right at the bottom. In the diagram on the left, we can cancel the $1$-- and $2$--handle pairs easily for all $t$, since there is one dotted circle with only one $2$--handle running over it, and we start the cancellation there. Thus in the end $\wat(\Theta)$ is equal to the mapping class given by the diagram on the bottom right.
 \begin{figure}
  \labellist
  \large\hair 2pt
  \pinlabel $=$ at 145 160
  \pinlabel $=$ at 20 50
  \pinlabel $+$ at 132 50
  \small
  \pinlabel $A$ [r] at 13 120
  \endlabellist
  \centering
  \includegraphics[width=.8\textwidth]{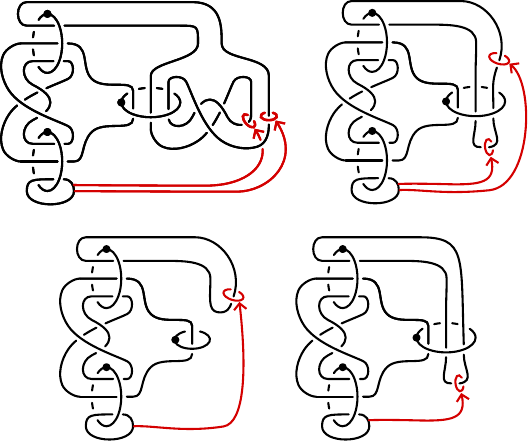}
  \caption{Split the spinning sphere for $A$, isotope and then decompose additively in $\Mod(S^4)$. \label{F:SplitASpin}}
 \end{figure}
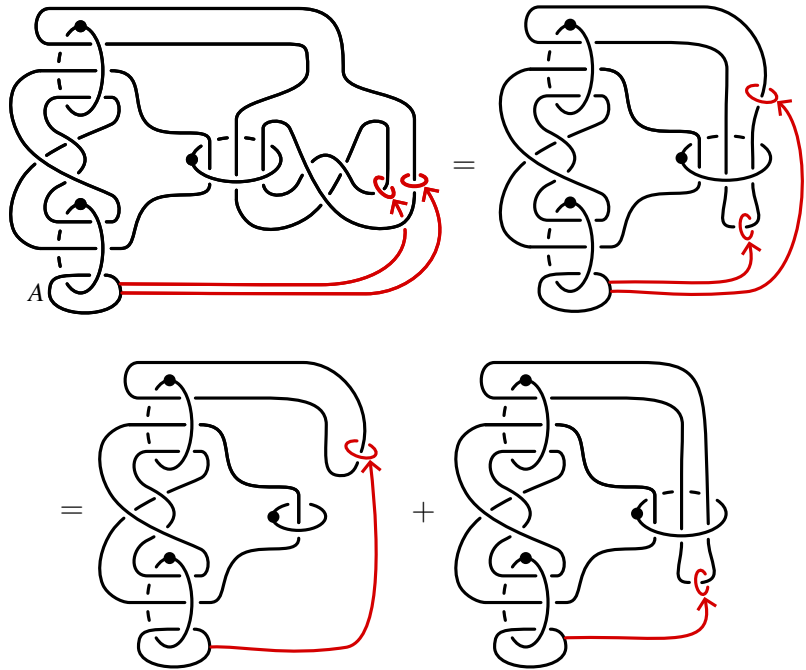

 In Figure~\ref{F:SlideBOverA}, we start with the diagram from the bottom right of Figure~\ref{F:SplitASpin}, slide $2$--handle $B$ twice over $2$--handle $A$ to get $B$ off of $a$, and then cancel $a$ and $A$. This is followed by an isotopy and then, again, a decomposition in $\Mod(S^4)$. The diagram on the bottom left easily cancels, so now $\wat(\Theta)$ is described by the diagram on the bottom right.
 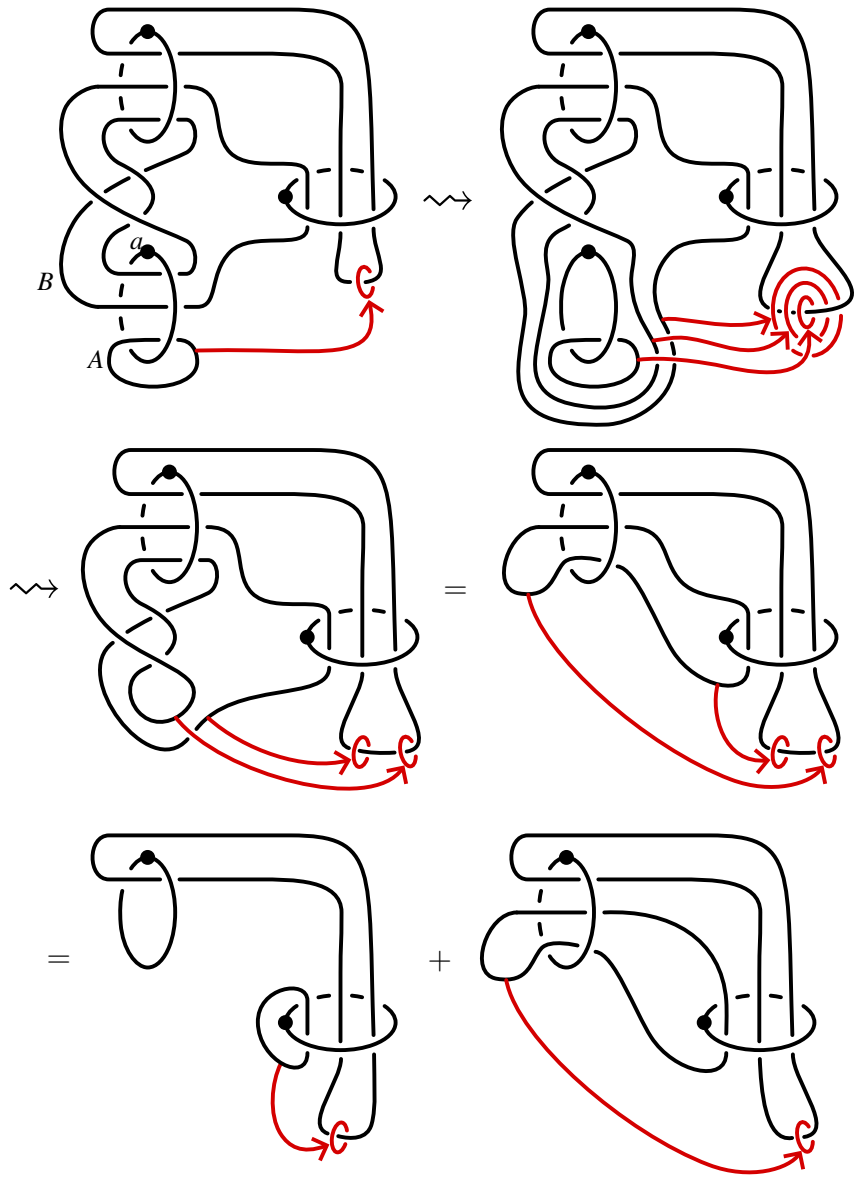
\begin{figure}
  \labellist
  \hair 2pt
  \huge\pinlabel $\rightsquigarrow$ at 100 250
  \pinlabel $\rightsquigarrow$ [r] at 2 150
  \large\pinlabel $=$ at 102 150
  \pinlabel $=$ at 0 55
  \pinlabel $+$ at 98 55
  \small
  \pinlabel $A$ [r] at 13 210
  \pinlabel $B$ [r] at 0 230
  \pinlabel $a$ [b] at 20 237
  \endlabellist
  \centering
  \includegraphics[width=.8\textwidth]{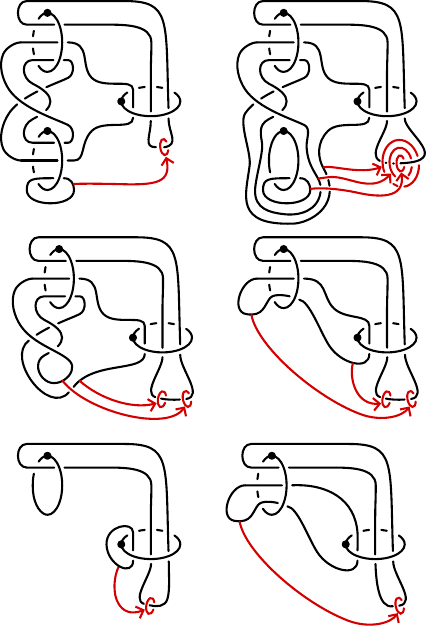}
  \caption{Slide $B$ over $A$, cancel $a$ and $A$, isotope, and then decompose. \label{F:SlideBOverA}}
 \end{figure}
  
 Next, in Figure~\ref{F:SlideCOverB}, we start with the diagram from the bottom right of Figure~\ref{F:SlideBOverA}, slide $C$ twice over $B$ to get $C$ off of $b$, and then cancel $B$ and $b$. Again we follow this with an isotopy and a decomposition, expressing $\wat(\Theta)$ as a sum of two elements of $\Mod(S^4)$, one of which immediately cancels, leaving $\wat(\Theta)$ equal to the mapping class described by the diagram on the bottom right.
 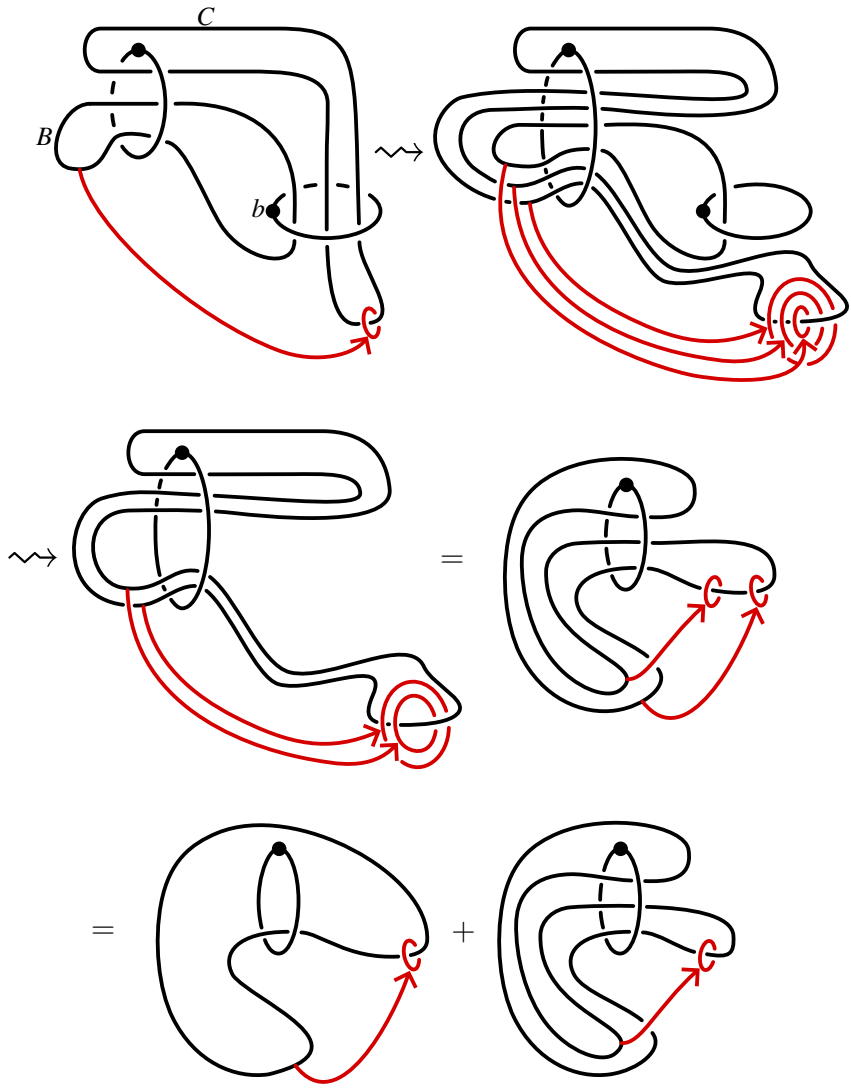
\begin{figure}
  \labellist
  \hair 2pt
  \huge\pinlabel $\rightsquigarrow$ at 91 245
  \pinlabel $\rightsquigarrow$ [r] at 3 138
  \large\pinlabel $=$ at 105 138
  \pinlabel $=$ at 13 40
  \pinlabel $+$ at 108 40
  \small
  \pinlabel $C$ [b] at 40 278
  \pinlabel $B$ [r] at 1 250
  \pinlabel $b$ [r] at 57 231
  \endlabellist
  \centering
  \includegraphics[width=.8\textwidth]{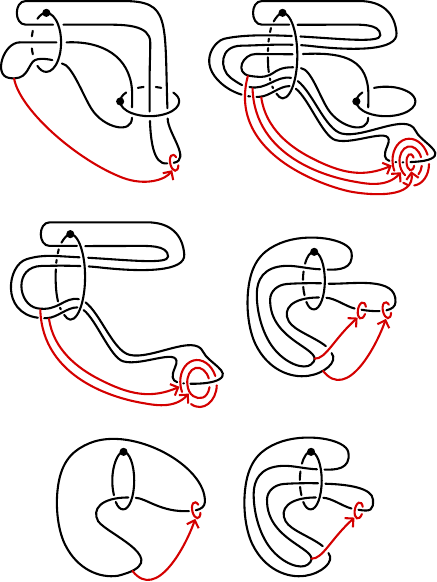}
  \caption{Slide $C$ over $B$, cancel $b$ and $B$, isotope, and then decompose. \label{F:SlideCOverB}}
 \end{figure}
  
 Finally, Figure~\ref{F:FinalIsotopy} shows an isotopy from the diagram on the bottom right of Figure~\ref{F:SlideCOverB} to our target diagram on the right of Figure~\ref{F:BeginningAndEnd}, which is a diagram for $\cH(\alpha(1))$.
 \begin{figure}
  \labellist
  \hair 2pt
  \large\pinlabel $=$ at 72 35
  \pinlabel $=$ at 136 35
  \endlabellist
  \centering
  \includegraphics[width=.6\textwidth]{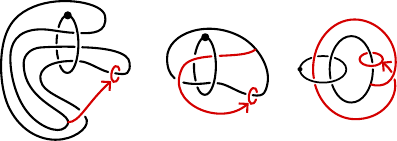}
  \caption{The final isotopy. \label{F:FinalIsotopy}}
 \end{figure}

\end{proof}

%
%
%
\bibliographystyle{amsalpha}
%

%

\bibliography{ThetaDiffeo}

\end{document}